\documentclass[10pt]{article}

\usepackage{amstext}
\usepackage{amssymb}
\usepackage{amsmath}
\usepackage{amsthm}
\usepackage{amsfonts,amsthm,latexsym,amsmath,amssymb,amscd,amsmath,epsf}
\usepackage{fullpage}
\usepackage[dvips]{graphicx}        
\graphicspath{{./pict/}}            

\usepackage[cp1251]{inputenc}
\usepackage[russian,english]{babel}

\theoremstyle{plain}
\newtheorem{theorem}{Theorem}

\theoremstyle{definition}

\renewcommand{\Im}{\operatorname{Im}}
\renewcommand{\Re}{\operatorname{Re}}

\begin{document}
\vspace{\baselineskip} \thispagestyle{empty}

\begin{center}
\Large{On the theory of symmetric polynomials\footnote{The original paper was published in \emph{Math. Sb., vol. 40, no. 3, 1933, pp. 271-283}. Translated from the Russian original by Mikhail Tyaglov.}}\\

\vspace{10mm}

\small{by Mark Krein (Odessa)}\\
\end{center}

\vspace{10mm}

\setcounter{equation}{0}
\section*{Introduction}

If $g(x,y)$ is a finite or infinite expression depending on $x$ and $y$
\begin{equation*}
g(x,y)=\sum_{i,k=0}^{m}g_{ik}x^iy^k\qquad\quad(g_{ik}=\overline{g}_{ik},\ m\leqslant\infty),
\end{equation*}
then by the symbol
\begin{equation*}
[g(x,y)]_n
\end{equation*}
we denote the Hermitian form
\begin{equation*}
\sum_{i,k=0}^{n-1}g_{ik}x_i\overline{x}_k
\end{equation*}
(where the bar denotes the complex conjugation).

As is known, due to the works by Schur~\cite{Schur}, Cohn~\cite{Cohn}, and due to additional notice by Fujiwara~\cite{Fujiwara},
the following theorem was established.

\emph{If the Hermitian form}\footnote{Here $\overline{f}(x)$ denotes the expression whose coefficients are complex conjugate to ones of~$f(x)$.}
\begin{equation*}
\begin{array}{c}
\displaystyle\mathfrak{H}[g;x_0,x_1,x_2,\ldots,x_{n-1}]=\left[\dfrac{g^{*}(x)\overline{g^{*}}(y)-g(x)\overline{g}(y)}{1-xy}\right]_n=
\\
=\displaystyle\sum_{\lambda=0}^{n-1}\left|\overline{a}_nx_{\lambda}+\overline{a}_{n-1}x_{\lambda+1}+\ldots+\overline{a}_{\lambda}x_{n}\right|^2-
\sum_{\lambda=0}^{n-1}\left|a_0x_{\lambda}+a_{1}x_{\lambda+1}+\ldots+a_{n-\lambda}x_{n}\right|^2
\end{array}
\end{equation*}
\emph{constructed from the polynomials}
\begin{equation*}
g(x)=a_0+a_1x+\ldots+a_nx^n,
\end{equation*}
\begin{equation*}
g^{*}(x)=x^n\overline{g}\left(\dfrac1{x}\right)=\overline{a}_n+\overline{a}_{n-1}x+\ldots+\overline{a}_0x^n,
\end{equation*}
\emph{has $\pi$ positive and $\nu$ negative squared terms, and the dimension of its kernel is $d$ $(\pi+\nu+d=n)$, then the polynomials $g(x)$ and $g^{*}(x)$ have
the greatest common divisor $D(x)$ of degree~$d$, and the polynomial $\dfrac{g(x)}{D(x)}$ has $\pi$ roots inside the circle $|x|=1$
and $\nu$ roots outside it.}

Therefore, to count exactly the number of roots of $g(x)$ inside the circle $|x|=1$, we should be able to count the number of roots of $D(x)$ inside this
circle.

At the same time, the polynomial $D(x)$ with the proper normalization by a constant factor is a symmetric polynomial, that is,
\begin{equation*}
D(x)=D^{*}(x)=x^n\overline{D}\left(\dfrac1x\right).
\end{equation*}

Meanwhile, for symmetric polynomials A.\,Cohn established the following theorem.

\vspace{2mm}
\textbf{Cohn's Theorem.} \textit{The number of roots of a symmetric polynomial inside the circle $|x|=1$ equals the number of roots of $g'(x)$ outside this circle.}
\vspace{2mm}

But the derivative $g'(x)$ of a symmetric polynomials is not a symmetric polynomial anymore, so one can apply Schur-Cohn's theorem
to $g'(x)$. Thus, this theorem together with Cohn's theorem completely solves the posed problem.

A.\,Cohn proved his theorem with the help of Rouch\'e's theorem and a number of laborious considerations ``by continuity''.

While Schur-Cohn's theorem can be established by the pure algebraic technique of Li\'enard and Chipart~\cite{Lienard_Chipart} (see also Fujiwara~\cite{Fujiwara}),
we do not know whether the same
was done for Cohn's theorem.

In Section~\ref{section.1} of this note, we construct two quadratic forms which allow us to count the number of roots of a symmetric polynomial
inside the circle $|x|=1$. Comparing one of those forms with the form $\mathfrak{H}$, we obtain Cohn's theorem and also
its generalization (see Theorem~\ref{Theorem.3}) pure algebraically.

In Section~\ref{section.2}, we give a criterion for two symmetric polynomial to have interlacing roots on the circle $|x|=1$. We also prove
a theorem analogous to V.\,Markov's theorem~\cite{MarkovV} (see also~\cite{Grave}) and a series of some other statements.

In Section~\ref{section.3}, we discuss analogies between symmetric and real polynomials.

\section{ }\label{section.1}

Let us denote by $s_k$ $(k=0,\pm1,\pm2,\ldots)$ the sum of $k^{\text{th}}$ powers of roots of a polynomial~$g(x)$.

It is easy to see that a polynomial $g(x)$ is symmetric or differs from symmetric by a constant factor if, and only if,
together with a root $\alpha$, $|\alpha|\neq1$, the polynomial $g(x)$ has the root $\alpha^{*}=\dfrac1{\overline{\alpha}}$ of
the same multiplicity as the root~$\alpha$.

This implies that $s_{-k}=\overline{s}_{k}$ $(k=0,1,2,\ldots)$ for every symmetric polynomial. Indeed, let $\varepsilon_1$, $\varepsilon_2$, \ldots,
$\varepsilon_p$ be all distinct roots of $g(x)$ with absolute value equal to $1$. Let their multiplicities equal $\rho_1$, $\rho_2$, \ldots,
$\rho_p$, respectively. Furthermore, let $\beta_1$, $\beta^{\,*}_1$, $\beta_2$, $\beta^{\,*}_2$, \ldots, $\beta_q$, $\beta^{\,*}_q$ be all
the distinct pairs of roots symmetric w.r.t. the circle $|x|=1$ $\left(\beta^{*}_i=\dfrac1{\overline{\beta}_i}\right)$ with correspondent
multiplicities $\sigma_1$, $\sigma_2$, \ldots, $\sigma_q$.

So for $s_k$ we have the expression
\begin{equation}\label{s_k}
s_k=\sum_{s=1}^p\rho_s\varepsilon_s^k+\sum_{t=1}^q\sigma_t\left(\beta_t^{\,k}+\beta^{\,*k}_t\right).
\end{equation}
On the other hand, $\overline{\varepsilon}_s^{\,k}=\varepsilon_s^{-k}$, $\overline{\beta}_t^{\,k}=\beta_t^{\,*-k}$ $(s=1,2,\ldots,p;t=1,2,\ldots,q)$. Therefore,
in fact,
\begin{equation*}
s_{-k}=\overline{s}_k.
\end{equation*}

Now we prove the following theorem.
\begin{theorem}\label{Theorem.1}
If the Hermitian form
\begin{equation*}
\mathfrak{S}=\sum_{i,k=0}^{n-1}s_{i-k}x_i\overline{x}_k,
\end{equation*}
constructed for a given symmetric polynomial $g(x)$ has $\pi$ positive squared terms and $\nu$ negative squared terms, then the polynomial
$g(x)$ has $\pi-\nu$ distinct roots $\varepsilon_i$ with absolute values equal to $1$ and $\nu$ distinct pairs $\beta_i$, $\beta_i^{*}$
symmetric w. r. t. the circle $|x|=1$.
\end{theorem}
\noindent\textit{Remark.} Thus, Theorem~\ref{Theorem.1} implies that the greatest common divisor $D(x)$ of polynomials $g(x)$ and $g'(x)$ has degree $d$ which is equal to
dimension~$d$ $(d\geqslant0)$ of the kernel of the form~$\mathfrak{S}$ $(\pi+\nu+d=n)$.

\begin{proof}
According to the formula~\eqref{s_k}, we obviously have
\begin{equation*}
s_{i-k}=\sum_{s=1}^p\rho_s\varepsilon_s^{i-k}+\sum_{t=1}^q\sigma_t\left[\beta_t^{i-k}+\beta^{*i-k}_t\right]=
\sum_{s=1}^p\rho_s\varepsilon_s^{i}\overline{\varepsilon}_s^{\,k}+\sum_{t=1}^q\sigma_t\left[\beta_t^{i}\overline{\beta_t^{\,*k}}+\overline{\beta_t^{\,k}}\beta^{*i}_t\right].
\end{equation*}
Hence
\begin{equation*}
\begin{array}{c}
\displaystyle\mathfrak{S}=\sum s_{i-k}x_i\overline{x}_k=\sum_{s=1}^p\rho_s|x_0+x_1\varepsilon_s+x_2\varepsilon_s^2+\ldots+x_{n-1}\varepsilon_s^{n-1}|^2+\\
\displaystyle+\sum_{t=1}^q\sigma_t(x_0+x_1\beta_t+x_2\beta_t^{\,2}+\ldots+x_{n-1}\beta_t^{\,n-1})(\overline{x}_0+\overline{x}_1\overline{\beta}_t^{\,*}+\ldots+\overline{x}_{n-1}\overline{\beta}_t^{\,*n-1})+\\
\displaystyle+\sum_{t=1}^q\sigma_t(x_0+x_1\beta_t^{\,*}+\ldots+x_{n-1}\beta_t^{\,*n-1})(\overline{x}_0+\overline{x}_1\overline{\beta}_t+\ldots+\overline{x}_{n-1}\overline{\beta}_t^{\ n-1})
\end{array}
\end{equation*}
Putting now
\begin{equation*}
\begin{array}{l}
\,\ X_s=x_0+x_1\varepsilon_s+\ldots+x_{n-1}\varepsilon_s^{n-1}\qquad \quad\ \, (s=1,2,\ldots,p),\\
\left.
\begin{array}{l}
Y_t=x_0+x_1\beta_t+\ldots+x_{n-1}\beta_t^{\,n-1},\\
Z_t=x_0+x_1\beta_t^{\,*}+\ldots+x_{n-1}\beta_t^{\,*n-1},
\end{array}\right\}\quad(t=1,2,\ldots,q),
\end{array}
\end{equation*}
we obtain
\begin{equation*}
\begin{array}{c}
\displaystyle\mathfrak{S}=\sum_{s=1}^p\rho_s|X_s|^2+\sum_{t=1}^q\sigma_t(Y_t\overline{Z}_t+Z_t\overline{Y}_t)=\\
\displaystyle=\sum_{s=1}^p\rho_s|X_s|^2+\sum_{t=1}^q2\sigma_t|U_t|^2-\sum_{t=1}^q2\sigma_t|V_t|^2,
\end{array}
\end{equation*}
where
\begin{equation*}
U_t=\dfrac{Y_t+Z_t}2,\quad V_t=\dfrac{Y_t-Z_t}2.
\end{equation*}

From the last representation of the form $\mathfrak{S}$ as the sum of independent positive and negative squared terms, we deduce that
$\nu=q$ and $\pi=p+q$, so $p=\pi-\nu$, as required.
\end{proof}

In addition to the form $\mathfrak{S}$, one can also construct another form that allows us to count the number of roots of the symmetric polynomial
$g(x)$ inside the circle $|x|=1$.

In order to do this, let us consider the function
\begin{equation*}
F(x)=\dfrac n2-\dfrac{xg'(x)}{g(x)}=\dfrac{g_{\delta}(x)}{g(x)},
\end{equation*}
where, thus, in what follows, by $g_{\delta}(x)$ we denote the polynomial
\begin{equation*}
g_{\delta}(x)=\dfrac n2g(x)-xg'(x).
\end{equation*}
It is easy to see that
\begin{equation*}
F(x)=\dfrac{s_0}2+\sum_{i=1}^{+\infty}s_{-i}x^i.
\end{equation*}

As well, it is not difficult to check the identity
\begin{equation*}
G(x,y)=\dfrac{F(x)+\overline{F}(y)}{1-xy}=\sum_{i,k=1}^{+\infty}s_{i-k}y^ix^k,
\end{equation*}
from which, by multiplying both sides of this identity by $g(x)\overline{g}(y)$, we find that
\begin{equation}\label{(2)}
\begin{array}{c}
\displaystyle\dfrac{g(x)\overline{g}_{\delta}(y)+g_{\delta}(x)\overline{g}(y)}{1-xy}=g(x)\overline{g}(y)G(x,y)=\\
\displaystyle=\sum_{i,k=1}^{+\infty}s_{i-k}(a_0x^k+a_1x^{k+1}+\ldots+a_nx^{n+k})(\overline{a}_0y^i+\overline{a}_1y^{i+1}+\ldots+\overline{a}_ny^{n+i}).
\end{array}
\end{equation}

But if $g(x)=a_0+a_1x+\ldots+a_nx^n$ is a symmetric polynomial $(a_k=\overline{a}_{n-k})$, then the polynomial~$g_{\delta}(x)$ is skew-symmetric, that is, the following holds
\begin{equation*}
g^{*}_{\delta}(x)=x^n\overline{g}_{\delta}\left(\dfrac1x\right)=-g_{\delta}(x).
\end{equation*}
Due to this property the left part of~\eqref{(2)} is a polynomial of $x$ and $y$. So, let
\begin{equation}\label{(3)}
\dfrac{g(x)\overline{g}_{\delta}(y)+g_{\delta}(x)\overline{g}(y)}{1-xy}=\sum_{i,k=0}^{n-1}a_{ik}x^iy^k.
\end{equation}
Obviously, the identity~\eqref{(2)} remains the same if we change $x^i$ to $x_i$ and $y^k$ to~$\overline{x}_k$.
Then according to~\eqref{(3)}, we obtain
\begin{equation}\label{(4)}
\sum_{i,k=0}^{+\infty}s_{i-k}z_k\overline{z}_i=\sum_{i,k=0}^{n-1}a_{ik}x_i\overline{x}_k
\end{equation}
if we set
\begin{equation*}
z_k=a_0x_k+a_1x_{k+1}+\ldots+a_nx_{n+k}\qquad (k=0,1,2,\ldots,\infty).
\end{equation*}

Now putting $x_n=x_{n+1}=\ldots=0$ in~\eqref{(4)} we get
%
$$
\sum_{i,k=0}^{n-1}s_{i-k}\overline{z}_iz_k=\sum_{i,k=0}^{n-1}a_{ik}x_i\overline{x}_k,\eqno(4a)
$$
%
where
\begin{equation*}
\begin{array}{c}
z_0\ \ \ =\ a_0x_0+a_1x_{1}+\ldots+a_{n-1}x_{n-1},\\
z_1\ \ \ =\ \,\quad\qquad a_1x_{0}+\ldots+a_{n-2}x_{n-1},\\
\cdots\cdots\cdots\cdots\cdots\cdots\cdots\cdots\cdots\cdots\cdots\cdots\cdots\\
z_{n-1}=\ \,\,\,\,\quad\qquad\qquad\qquad\qquad a_0x_{n-1}.
\end{array}
\end{equation*}

Since for the symmetric polynomial $g(x)$, we have $a_0=\overline{a}_n\neq0$, the transformation above is non-singular.

Thus, the identity~$(4\mathrm{a})$ shows that the form $\mathfrak{S}$ in Theorem~\ref{Theorem.1} can be changed by
the form
\begin{equation*}
\mathfrak{K}=\sum a_{ik}x_i\overline{x}_k
\end{equation*}

So, the following theorem holds.
\begin{theorem}\label{Theorem.2}
If the Hermitian form
\begin{equation*}
\mathfrak{K}[g;x_0,x_1,\ldots,x_{n-1}]=\left[\dfrac{g(x)\overline{g}_{\delta}(y)+g_{\delta}(x)\overline{g}(y)}{1-xy}\right]_n,
\end{equation*}
constructed for a given symmetric polynomial $g(x)$ has $\pi$ positive squared terms and $\nu$~negative squared terms, then the polynomial
$g(x)$ has $\pi-\nu$ distinct roots $\varepsilon_i$ with absolute values equal to $1$ and $\nu$ distinct pairs $\beta_i$, $\beta_i^{*}$
symmetric w. r. t. the circle $|x|=1$.
\end{theorem}

This theorem can also be proved by the method of Li\'enard and Chipart~\cite{Lienard_Chipart} taking into account the fact that
if $g=g_1g_2$, then $g_{\delta}=g_{1\delta}g_2+g_1g_{2\delta}$, but here we will not go into details.

Now we prove the following theorem.
\begin{theorem}\label{Theorem.3}
If $g(x)$ is a symmetric polynomial, then for $\Re z>0$, the polynomial
\begin{equation*}
f(x)=g_{\delta}(x)-zg(x),
\end{equation*}
has as many roots outside the circle $|x|=1$ as many roots of the polynomial $g(x)$ lie inside or outside this circle.
\end{theorem}

Before we prove this theorem, let us note that A.\,Cohn's theorem is a consequence of this theorem corresponding to $z=\dfrac n2$, since in this case,
\begin{equation*}
f(x)=-xg'(x).
\end{equation*}
\begin{proof}
Let us construct Schur-Cohn's form
\begin{equation*}
\mathfrak{H}[f;x_0,x_1,\ldots,x_{n-1}]=\left[\dfrac{f^{*}(x)\overline{f^{*}}(y)-f(x)\overline{f}(y)}{1-xy}\right]_n
\end{equation*}
for the polynomial $f(x)$. Since
\begin{equation}\label{(5)}
\left.\begin{array}{c}
f(x)=g_{\delta}(x)-zg(x),\\
f^{*}(x)=-g_{\delta}(x)-\overline{z}g(x),
\end{array}\right\}
\end{equation}
after a simple calculation we find that
\begin{equation}\label{(6)}
\mathfrak{H}[f;x_0,x_1,\ldots,x_{n-1}]=2\xi\mathfrak{K}[g;x_0,x_1,\ldots,x_{n-1}],
\end{equation}
where $\xi+i\eta=z$. By~\eqref{(5)} the greatest common divisor $D(x)$ of the polynomials $f$ and $f^{*}$
is the greatest common divisor of the polynomials $g(x)$ and $g_{\delta}(x)$ and, consequently, of the polynomials $g(x)$ and $g'(x)$.
Now the identity~\eqref{(6)} establishes our theorem.

Indeed, by Schur-Cohn's theorem the number $\nu_{\mathfrak{H}}$ of negative squared terms of the form~$\mathfrak{H}$
equals the number of roots of $\dfrac{f}{D}$ outside the unit disk. But according to the identity~\eqref{(6)}, for $\xi=\Re z>0$, this number
equals $\nu_{\mathfrak{K}}$, so by Theorem~\ref{Theorem.2} it equals the number of roots of $\dfrac{g}{D}$ inside (or outside) the circle $|x|=1$.
\end{proof}

It can be analogously proved that if $\xi=\Re z<0$, then the polynomial
\begin{equation*}
f(x)=g_{\delta}(x)-zg(x)
\end{equation*}
has the same number of roots inside the unit disk as the polynomial $g(x)$ does. These conclusions have also the following curious
interpretation.

\textit{If $g(z)$ is a symmetric polynomial, then the function
\begin{equation*}
z=\dfrac{g_{\delta}(x)}{g(x)}
\end{equation*}
maps the disk $|x|<1$ into a domain consisting of $k$ sheets with $\Re z<0$ and $n-k$ sheets with $\Re z>0$, where $k$ is the number
of poles of the function $z$ in the disk $|x|<1$.}

We conclude this Section noticing that it is not difficult to separate a positive squared term from the form $\mathfrak{K}$,
that is, it is easy to check that
\begin{equation*}
\mathfrak{K}[g;x_0,x_1,\ldots,x_{n-1}]=\dfrac1n\mathfrak{H}[g';x_0,x_1,\ldots,x_{n-1}]+\dfrac1n[g'(x)\overline{g}'(y)]_n.
\end{equation*}
This again implies A.\,Cohn's theorem.

\section{ }\label{section.2}

Let us correspond to every symmetric polynomial $g(x)$ the following $2\pi$-periodic
function
\begin{equation*}
G(\varphi)=e^{-\tfrac{ni\varphi}2}g\left(e^{i\varphi}\right).
\end{equation*}
Then it is easy to see that
\begin{equation*}
G'(\varphi)=-ie^{-\tfrac{ni\varphi}2}g_{\delta}\left(e^{i\varphi}\right).\eqno(\mathrm{A})
\end{equation*}

We also agree that two symmetric polynomials $g(x)$ and $h(x)$ have \textit{interlacing} roots
if all their roots are simple, lie on the circle $|x|=1$, and between any two consecutive roots of
one polynomial there lies one, and only one, root of the second polynomial.

It is evident that polynomials $g(x)$ and $h(x)$ have interlacing roots if, and only if,
$g(x)$ and $h(x)$ are of the same degree, and the function
\begin{equation*}
H(\varphi)=e^{-\tfrac{ni\varphi}2}h\left(e^{i\varphi}\right)
\end{equation*}
has different signs at any two consecutive roots of $g(x)$.

If all the roots of the symmetric polynomial $g(x)$ are simple and lie on the circle $|x|=1$, then by Rolle's theorem,
the formula $(\mathrm{A})$ implies that the roots of $g_{\delta}(x)$ interlace the roots of $g(x)$.

Furthermore, this implies that the roots of the polynomial $h(x)$ interlace the roots of such a polynomial $g(x)$
if, and only if, the ratio
\begin{equation*}
\dfrac{H(\varphi)}{G'(\varphi)}=i\dfrac{h(\varphi)}{g_{\delta}(\varphi)}\qquad (x=e^{i\varphi})
\end{equation*}
is of the same sign at all roots of the polynomial $g(x)$. This remark allows us to prove the following theorem.
\begin{theorem}\label{Theorem.4}
Two symmetric polynomials $g(x)$ and $h(x)$ have interlacing roots if and only if the form
\begin{equation*}
\mathfrak{K}[g,h;x_0,x_1,\ldots,x_{n-1}]=\left[i\dfrac{g(x)\overline{h}(y)-\overline{g}(y)h(x)}{1-xy}\right]_n
\end{equation*}
is sign-definite.
\end{theorem}
\begin{proof}
We first prove the sufficiency.

If the form $\mathfrak{K}$ is of a definite sign, then the expression
\begin{equation}\label{(7)}
\mathfrak{K}[g,h;1,\alpha,\ldots,\alpha^{n-1}]=i\dfrac{g(\alpha)\overline{h(\alpha)}-\overline{g(\alpha)}h(\alpha)}{1-\alpha\overline{\alpha}}
\end{equation}
for all $\alpha$, $|\alpha|\neq1$ preserves its sign. Therefore, $g(\alpha)\neq0$ for $|\alpha|\neq1$ that follows from the right hand side of~\eqref{(7)}.

Thus, $g(x)$ has all its roots lying on the circle $|x|=1$.

On the other hand, if $|\alpha|=1$, then it is easy to see that
\begin{equation}\label{(8)}
\mathfrak{K}[g,h;1,\alpha,\ldots,\alpha^{n-1}]=-i\dfrac{g(\alpha)\overline{h'(\alpha)}-\overline{g'(\alpha)}h(\alpha)}{\alpha},
\end{equation}
since $\overline{\alpha}=\dfrac1{\alpha}$ in this case, so by L'H\^opital's rule,
\begin{equation*}
\lim_{y\to\tfrac1x}\left[i\dfrac{g(x)\overline{h}(y)-\overline{g}(y)h(x)}{1-xy}\right]=-i\dfrac{g(x)\overline{h'}\left(\dfrac1x\right)-\overline{g'}\left(\dfrac1x\right)h(x)}{x}.
\end{equation*}

The expression~\eqref{(8)} can also be transformed as follows
\begin{equation}\label{(9)}
\begin{array}{c}
\mathfrak{K}[g,h;1,\alpha,\ldots,\alpha^{n-1}]=-i[g(\alpha)\overline{\alpha h'(\alpha)}-h(\alpha)\overline{\alpha g'(\alpha)}]=\\
\\
=i[g(\alpha)\overline{h_{\delta}(\alpha)}-h(\alpha)\overline{g_{\delta}(\alpha)}],
\end{array}
\end{equation}
since
\begin{equation*}
\dfrac{n}2[h(\alpha)\overline{g(\alpha)}-g(\alpha)\overline{h(\alpha)}]=\dfrac{n}2\alpha^{-n}[h(\alpha)g(\alpha)-g(\alpha)h(\alpha)]=0.
\end{equation*}
Inasmuch as $\mathfrak{K}$ does not change its sign, from~\eqref{(8)} it is follows that $g(\alpha)$ and $g'(\alpha)$ are non-zero simultaneously, that is,
all the roots of $g(x)$ are simple. On the other hand, putting in~\eqref{(9)} $\alpha=\alpha_k$ $(k=1,2,\ldots,n)$, where $\alpha_k$ is a root
of $g(x)$, we find that
$$
\mathfrak{K}[g,h;1,\alpha,\ldots,\alpha^{n-1}]=-i h(\alpha_k)\overline{g_{\delta}(\alpha_k)}=
-i\dfrac{h(\alpha_k)}{g_{\delta}(\alpha_k)}g_{\delta}(\alpha_k)\overline{g_{\delta}(\alpha_k)}.\eqno{(\mathrm{B})}
$$
From this formula we infer that the expression $i\dfrac{h(\alpha_k)}{g_{\delta}(\alpha_k)}$ has the same sign
for all $\alpha_k$. According to the remark before Theorem~\ref{Theorem.4} this implies the sufficiency of the statement
of Theorem~\ref{Theorem.4}.

We now prove the necessity. Let the roots of $g(x)$ and $h(x)$ be interlacing. Let $\alpha_1$, $\alpha_2$, \ldots, $\alpha_n$ $(|\alpha_i|=1)$ be
all distinct roots of $g(x)$. We denote by
\begin{equation*}
\mathfrak{K}\left[g,h;\begin{smallmatrix}x_0,x_1,\ldots,x_{n-1}\\y_0,y_1,\ldots,y_{n-1}\end{smallmatrix}\right]
\end{equation*}
the bilinear form $\sum a_{ik}x_k\overline{y}_k$ corresponding to the Hermitian form
\begin{equation*}
\sum_{i,k=0}^{n-1}a_{ik}x_i\overline{x}_k=\mathfrak{K}[g,h;x_0,x_1,\ldots,x_{n-1}].
\end{equation*}
It is easy to see that
\begin{equation*}
\mathfrak{K}\left[g,h;\begin{smallmatrix}1,\alpha_k,\ldots,\alpha_k^{n-1}\\1,\alpha_l,\ldots,\alpha_l^{n-1}\end{smallmatrix}\right]=0
\end{equation*}
for $k\neq l$. Therefore,
\begin{equation}\label{(10)}
\displaystyle\mathfrak{K}\left[g,h;\begin{smallmatrix}x_0,x_1,\ldots,x_{n-1}\\y_0,y_1,\ldots,y_{n-1}\end{smallmatrix}\right]=
\sum_{k=1}^{n}\dfrac{\mathfrak{K}\left[g,h;\begin{smallmatrix}x_0,x_1,\ldots,x_{n-1}\\1,\alpha_k,\ldots,\alpha_k^{n-1}\end{smallmatrix}\right]
\mathfrak{K}\left[g,h;\begin{smallmatrix}1,\alpha_k,\ldots,\alpha_k^{n-1}\\y_0,y_1,\ldots,y_{n-1}\end{smallmatrix}\right]}
{\mathfrak{K}[g,h;1,\alpha_k,\ldots,\alpha_k^{n-1}]},
\end{equation}
since the left hand side identically w.r.t. $x$ equals the right hand side for $n$ independent systems of particular values of $y$, namely, for
the following systems of values
\begin{equation*}
y_0=1,y_1=\alpha_k,\ldots,\alpha_k^{n-1},\qquad(k=1,2,\ldots,n).
\end{equation*}
The identity~\eqref{(10)} for $y_i=x_i$ $(i=0,1,\ldots,n-1)$ gives us an expansion of our form into the sum of squared terms with coefficients, which
are of the same sign by $(\mathrm{B})$ and by the conditions that $g$ and $h$ satisfy. Theorem is proved.
\end{proof}

Let us now establish the following curious statement.
\begin{theorem}\label{Theorem.5}
If the roots of symmetric polynomials $g(x)$ and $h(x)$ are interlacing, then the roots of the polynomials $g_{\delta}(x)$ and $h_{\delta}(x)$
interlace, as well.
\end{theorem}
\begin{proof}
In fact, if the roots of the polynomials $g(x)$ and $h(x)$ are interlacing, then as we proved above, the expression
\begin{equation}\label{(11)}
i[g(\alpha)\overline{h_{\delta}(\alpha)}-h(\alpha)\overline{g_{\delta}(\alpha)}]
\end{equation}
is real and does not change its sign on the circle $x=e^{i\varphi}$.

Let us consider the functions
\begin{equation*}
H(\varphi)=e^{-\tfrac{ni\varphi}2}h\left(e^{i\varphi}\right)\quad\text{and}\quad G(\varphi)=e^{-\tfrac{ni\varphi}2}g\left(e^{i\varphi}\right).
\end{equation*}
Since the polynomials $g_{\delta}(x)$ and $h_{\delta}(x)$ are skew-symmetric, then we have
\begin{equation*}
H'(\varphi)=-ie^{-\tfrac{ni\varphi}2}h_{\delta}\left(e^{i\varphi}\right)=ie^{\tfrac{ni\varphi}2}\overline{h}_{\delta}\left(e^{-i\varphi}\right),
\end{equation*}
\begin{equation*}
G'(\varphi)=-ie^{-\tfrac{ni\varphi}2}g_{\delta}\left(e^{i\varphi}\right)=ie^{\tfrac{ni\varphi}2}\overline{g}_{\delta}\left(e^{-i\varphi}\right).
\end{equation*}
Consequently, the expression~\eqref{(11)} equals the following
\begin{equation}\label{(12)}
G(\varphi)H'(\varphi)-G'(\varphi)H(\varphi).
\end{equation}
Let now $\varphi_k$ and $\varphi_l$ be the arguments of two consecutive roots of $G'(\varphi)$. Putting into~\eqref{(12)} $\varphi=\varphi_k$ and
then $\varphi=\varphi_l$, we obtain that the expressions
\begin{equation*}
G(\varphi_k)H'(\varphi_k)\quad\text{and}\quad G(\varphi_l)H'(\varphi_l)
\end{equation*}
are of the same sign. But $G(\varphi_k)$ and $G(\varphi_l)$ obviously have different signs, consequently, $H'(\varphi_k)$
and $H'(\varphi_l)$ have different signs, as well. Therefore, between $\varphi_k$ and $\varphi_l$ there lies at least one root
of $H'(\varphi)$. Now the statement of the theorem follows from the fact that $G'(\varphi)$ and $H'(\varphi)$ have equal number
of roots.
\end{proof}

If the polynomials $g(x)$ and $h(x)$ are of an even degree, $n=2m$, then $G(\varphi)$ and $H(\varphi)$
are trigonometric sums of the form
\begin{equation*}
\text{I.}\quad a_0+\sum_{k=1}^{m}(a_k\cos k\varphi+b_k\sin k\varphi)
\end{equation*}

If the polynomials $g(x)$ and $h(x)$ are of an odd degree, $n=2m-1$, then $G(\varphi)$ and $H(\varphi)$
are trigonometric sums of the form
\begin{equation*}
\text{II.}\quad \sum_{k=1}^{m}\left[a_k\cos\left(k-\dfrac12\right)\varphi+b_k\sin\left(k-\dfrac12\right)\varphi\right].
\end{equation*}

Thus, our theorem can also be formulated as follows.

\textit{If roots of two trigonometric sums of type I or II are interlacing, then the roots of their derivatives are interlacing, as well.}

Also it is easy to prove the following theorem.
\begin{theorem}\label{Theorem.6}
If symmetric polynomials $g(x)$ and $h(x)$ have interlacing roots, then the roots of the symmetric polynomial
\begin{equation*}
f(x)=g(x)+th(x)\qquad(-\infty<t<\infty)
\end{equation*}
interlace the roots of the polynomials $g(x)$ and $h(x)$. The arguments of roots of $f(x)$ are monotone functions of parameter $t$.
\end{theorem}
\begin{proof}
The first part of the statement of the theorem follows from the formul\ae
\begin{equation*}
\begin{array}{l}
\mathfrak{K}[f,h;x_0,x_1,\ldots,x_{n-1}]=\quad\mathfrak{K}[g,h;x_0,x_1,\ldots,x_{n-1}],\\
\mathfrak{K}[f,g;x_0,x_1,\ldots,x_{n-1}]=-t\mathfrak{K}[g,h;x_0,x_1,\ldots,x_{n-1}].
\end{array}
\end{equation*}

To prove the second part of the statement we consider a root $\alpha$ of the function $f$ and differentiate by $t$
the left hand side of the equation
\begin{equation*}
f(\alpha)=g(\alpha)+th(\alpha)=0,
\end{equation*}
considering $\alpha$ as a function of $\varphi$ $(\alpha=e^{i\varphi})$, which, in its turn, depends on $t$. Then we obtain
\begin{equation*}
i\alpha f'(\alpha)\dfrac{d\varphi}{dt}=-h(\alpha).
\end{equation*}
Taking into account that $f(\alpha)=0$, this implies that
\begin{equation*}
-\dfrac{d\varphi}{dt}=i\dfrac{h(\alpha)}{f_{\delta}(\alpha)}.
\end{equation*}
Since the roots of $f(x)$ and $h(x)$ are interlacing, the right hand side of the last equality does not change its sign, as required.
\end{proof}

Concluding this Section we note that the interlacing criterion has a function-theoretical interpretation.
Namely, the following theorem holds.
\begin{theorem}\label{Theorem.7}
The roots of two symmetric polynomials $g(x)$ and $h(x)$ are interlacing if and only if the function
\begin{equation*}
z=\dfrac{h(x)}{g(x)}
\end{equation*}
maps the disk $|x|<1$ into one of two $n$-sheet half-planes, $\Im z>0$ or $\Im z<0$.
\end{theorem}
\begin{proof}
Indeed, let $z=\xi+i\eta$ $(\eta\neq0)$. We consider the function
\begin{equation*}
f(x)=h(x)-zg(x).
\end{equation*}
Evidently,
\begin{equation*}
f^{*}(x)=h(x)-\overline{z}g(x).
\end{equation*}
Easily, this implies that Schur-Cohn's form constructed for $f$ equals the following
\begin{equation*}
\mathfrak{H}[f;x_0,x_1,\ldots,x_{n-1}]=-2\eta\mathfrak{K}[g,h;x_0,x_1,\ldots,x_{n-1}].
\end{equation*}
Consequently, for the roots of $g(x)$ and $h(x)$ to be interlacing it is necessary and sufficient the form $\mathfrak{H}$ to be sign-definite.
In this case, by Schur-Cone's theorem, all the roots of $f(x)$ lie either inside the unit circle (if $\mathfrak{H}$ is positive definite) or
outside the unit circle (if $\mathfrak{H}$ is negative definite). This fact establishes the statement of the theorem.
\end{proof}

The last theorem shows that there is a close connection between the theory of symmetric polynomials and the well-known Carath\'eodory's
problem on positive harmonic functions~\cite{Caratheodory.1,{Caratheodory.2}}.

Note also that if we expand $\dfrac{ih(x)}{g(x)}$ into the series
\begin{equation*}
\dfrac{ih(x)}{g(x)}=\sigma_0+i\tau_0+\sum_{i=1}^{+\infty}\sigma_ix^i\qquad (\sigma_0\lessgtr0,\tau_0\lessgtr0),
\end{equation*}
then, analogously to what we do in Section~\ref{section.1}, instead of the form $\mathfrak{K}$ one can consider the form
\begin{equation*}
\mathfrak{H}(g,h;x_0,x_1,\ldots,x_{n-1})=\sum_{i,k=0}^{n-1}\sigma_{i-k}x_i\overline{x}_k\qquad (\sigma_{-i}=\overline{\sigma}_i),
\end{equation*}
which is connected to the form $\mathfrak{K}$ by the identity
\begin{equation*}
\mathfrak{K}[g,h;x_0,x_1,\ldots,x_{n-1}]=\sum_{i,k=0}^{n-1}\sigma_{i-k}\overline{z}_iz_k,
\end{equation*}
where
\begin{equation*}
\left.\begin{array}{c}
z_0\ \ \ =\ a_0x_0+a_1x_{1}+\ldots+a_{n-1}x_{n-1},\\
z_1\ \ \ =\ \,\quad\qquad a_1x_{0}+\ldots+a_{n-2}x_{n-1},\\
\cdots\cdots\cdots\cdots\cdots\cdots\cdots\cdots\cdots\cdots\cdots\\
z_{n-1}=\ \,\,\,\,\quad\qquad\qquad\qquad\qquad a_0x_{n-1}.
\end{array}\right\}\quad[g(x)=a_0+a_1x+\ldots+a_nx^n].
\end{equation*}
It is known that in the Carath\'eodory's problem, exactly the following Toeplitz forms are considered~\cite{Herglotz.2}\footnote{In this work, it is pointed out to the possibility
of establishing a criterion of root interlacing for two symmetric polynomials}.
\begin{equation*}
\sum\sigma_{i-k}\overline{z}_iz_k
\end{equation*}
%

\section{ }\label{section.3}

Since the transformation
\begin{equation*}
f(x)=(x-i)^ng\left(\dfrac{x+i}{x-i}\right)
\end{equation*}
transforms a symmetric polynomial $g(x)$ into a real one, there must exist some analogies between the
theory of symmetric and real polynomials.

In fact, Theorem~\ref{Theorem.1} corresponds to Borchardt's theorem~\cite{Borchardt.1,{Borchardt.2}}:

I. \textit{If the real form
\begin{equation*}
\sum_{i,k=0}^{n-1}s_{i+k}x_ix_k
\end{equation*}
constructed for a real polynomial $f(x)$ has $\pi$ positive squared terms and $\nu$ negative squared terms,
then the polynomial $f(x)$ has $\pi-\nu$ distinct real roots and $\nu$ distinct pairs of non-real roots}.

Using the technique of Section~\ref{section.1}, it is not difficult to transform this criterion into the following one.

II. \textit{If the form
\begin{equation*}
\left[\dfrac{f(x)f'(y)-f'(x)f(y)}{x-y}\right]_n=K[f;x_0,x_1,\ldots,x_{n-1}]
\end{equation*}
has $\pi$ positive squared terms and $\nu$ negative squared terms, then the polynomial $f(x)$ has $\pi-\nu$
distinct real roots and $\nu$ distinct pairs of non-real roots.}

Hermite~\cite{Hermite} suggested, instead of the form $K$, to consider the form
\begin{equation*}
K_1[f;x_0,x_1,\ldots,x_{n-1}]=\left[\dfrac{\check{f}(x)f'(y)-f'(x)\check{f}(y)}{x-y}\right]_n,
\end{equation*}
where
\begin{equation*}
\check{f}(x)=nf(x)-xf'(x).
\end{equation*}

Hermite's rule (we do not formulate it here) can be easily obtained from the rule II if
one takes into account the identity
\begin{equation*}
K[f;x_0,x_1,\ldots,x_{n-1}]=\dfrac1nK_1[f;x_0,x_1,\ldots,x_{n-1}]+\dfrac1n[f'(x)f'(y)]_n.
\end{equation*}

Theorem~\ref{Theorem.3} corresponds to the following theorem.

III. \textit{If $f(x)$ is a real polynomial, then the polynomial
\begin{equation*}
f'(x)-zf(x),
\end{equation*}
for $\Im z>0$, has as many roots in the half-plane $\Im x>0$ as many pairs of non-real roots the polynomial $f(x)$ has.}

We agree to say that roots of two real polynomials interlace if all their roots are real and distinct, and between any two
consecutive roots of one polynomial there lies one, and only one, root of the other polynomial. Then the analogue
of Theorem~\ref{Theorem.4} is the proposition that was established, in another form, by Hurwitz~\cite{Hurwitz}:

IV. \textit{Roots of two real polynomials $f$ and $F$ interlace if and only if the form
\begin{equation*}
\left[\dfrac{F(x)f(y)-F(y)f(x)}{x-y}\right]_n
\end{equation*}
is sign-definite.}

This proposition can certainly be proved by the same technique that we used in the proof of Theorem~\ref{Theorem.4}, but
one can find for it a very beautiful proof based on Sturm's theorem.

Theorem~\ref{Theorem.5} turns into Markov's theorem~\cite{MarkovV}:

V. \textit{If roots of polynomials $f(x)$ and $F(x)$ interlace, then the roots of their derivatives~$f'(x)$ and $F'(x)$ interlace, as well.}

It is obvious how one should formulate the analogue of Theorem~\ref{Theorem.6}.

Theorem~\ref{Theorem.7} turns into the following proposition.

VII. \textit{The roots of two real polynomials $f(x)$ and $F(x)$ are interlacing if, and only if, the function
\begin{equation*}
z=\dfrac{f(x)}{F(x)}
\end{equation*}
maps the half-plane $\Im x>0$ into one of two $n$-sheet half-planes, $\Im z>0$ or $\Im z<0$\footnote{Regarding Propositions III and VII, compare them with Chebotarev's work~\cite{Chebotarev}.}.}

For proving Propositions III and VII, it is better to use Hermite's theorem \footnote{See~\cite{Hermite}, pp.~41--44.}
(analogue of Schur-Cohn's theorem).

\textbf{Hermite's Theorem.} \textit{If the form
\begin{equation*}
\left[-i\dfrac{F(x)\overline{F}(y)-\overline{F}(x)F(y)}{x-y}\right]_n
\end{equation*}
has $\pi$ positive and $\nu$ negative squared terms, and if the dimension of its kernel is $d$
, then the polynomials $F(x)$ and $\overline{F}(x)$ have the greatest common divisor $D(x)$ of
degree $d$, and the polynomial $\dfrac{F(x)}{D(x)}$ has $\pi$ roots in the upper half-plane $\Im x>0$ and $\nu$ roots in  the lower half-plane $\Im x<0$.}

\vspace{5mm}

\textit{Remarks during proofread.} After this work was submitted to publication, the author got an opportunity to read the paper~\cite{Herglotz} by G.\,Herglotz,
where two forms constructed in Section~\ref{section.1} was also introduced, and Theorems~\ref{Theorem.1} and~\ref{Theorem.2} were proved. Since the technique
of the author is more elementary and completely differs from Herglotz's method which is connected to the theory of characteristics, the author allows himself
to leave Section~\ref{section.1} without any changes.


Unter Benutzung von elementaren Betrachtungen konstruiert der Verfasser zwei Formen, mit deren Hilfe die Wurzelanzahl des symmetrischen Polynoms $$f(x)=x^n\overline{f}\left(\dfrac1x\right)$$ innerhalb des Einheitskreises aufgez\"ahlt werden kann; zu diesen Formen kommt auch G.Herglotz, indem er von der Charakteristikentheorie ausgeht. Auf Grund dieser Resultate wird rein algebraisch der Satz von A.Cohn abgeleitet, nach dem das symmetrische Polynom dieselbe Anzahl von Wurzeln ausserhalb des Einheitskreises besitzt, wies eine Derivierte. Dieser Satz wird etwas verallgemeinert. Zum Schluss werden symmetrische Polynome mit sish trennenden Wurzeln betrachtet, und ein Analogon des W.Markowschen Satzes f\"ur symmetrische, also auch f\"ur trigonometrische Polynome, aufgestellt.

\end{document}